\documentclass[11pt]{article}
\usepackage[english]{babel}
\usepackage{epsfig}
\usepackage{amsmath}
\usepackage{amsthm}
\usepackage{amssymb}
 
\newtheorem{theorem}{Theorem}[section]
\newtheorem{definition}[theorem]{Definition}

\newtheorem{proposition}[theorem]{Proposition}
\newtheorem{corollary}[theorem]{Corollary}
\newtheorem{remark}[theorem]{Remark}

\title{\bf Recent and new results on octonionic Bergman and Szeg\"o kernels} 

\author{
Rolf S\"{o}ren Krau{\ss}har\\
Fachbereich Mathematik\\
Erziehungswissenschaftliche Fakult\"at\\
Universit\"at Erfurt\\
Nordh\"auser Str. 63\\
99089 Erfurt, Germany\\
soeren.krausshar@uni-erfurt.de }

\begin{document}
\maketitle
\begin{abstract}  
	Very recently one has started to study Bergman and Szeg\"o kernels in the setting of octonionic monogenic functions. In particular, explicit formulas for the Bergman kernel for the octonionic unit ball and for the octonionic right half-space as well as a formula for the Szeg\"o kernel for the octonionic unit ball have been established. In this paper we extend this line of investigation by developing explicit formulas for the Szeg\"o kernel of strip domains of the form ${\cal{S}} := \{z \in \mathbb{O} \mid 0 < \Re(z) < d\}$ from which we derive by a limit argument considering $d \to \infty$ the Szeg\"o kernel of the octonionic right half-space. Additionally, we set up formulas for the Bergman kernel of such strip domains and relate both kernels with each other. In fact, these kernel functions can be expressed in terms of one-fold periodic octonionic monogenic generalizations of the cosecant function and the cotangent function, respectively.
	     
\end{abstract}
{\bf Keywords}: octonions, Bergman kernel, Szeg\"o kernel, strip domains, octonionic trigonometric functions \\[0.1cm] 
\noindent {\bf Mathematical Review Classification numbers}: 30G35\\

\section{Introduction}  

Particularly during the last four years one notices a strongly growing interest in octonionic analysis, see for instance \cite{JRS,KO2018,KO2019,Kra2019-1, Kra2019-2, Kra2020-1,Nolder2018,WL2018,WL2020}. 
\par\medskip\par
One possible reason for this boost of interest has been the discovery that  the octonions and their function theory  play a key-role in the description of the symmetries arising in unifying physical models that connect the standard model of particle physics with aspects of supergravity,  see for example \cite{Burdik}.

\par\medskip\par

Already during the 1970s, P. Dentoni and M. Sce started to investigate how to carry over fundamental tools from Clifford analysis to the non-associative octonionic setting, cf. \cite{DS}. 

In fact, in spite of the non-associativity, it turned out to be possible to extend many analogues of classical fundamental theorems from the Clifford analysis setting to the octonionic setting. Key ingredients are for instance a Cauchy integral formula as well as Taylor and Laurent series representation formulas that are composed by octonionic versions of the Fueter polynomials, see for example \cite{Imaeda,Nono,XL2000,XL2001,XL2002,XZL}. Of course, one carefully has to put parenthesis in order to take the non-associative nature into account. 

Although some of these fundamental theorems formally look very similar to those that we have in the associative Clifford algebra setting, Clifford analysis and octononic analysis are really two essentially different function theories. 
Important substantial differences are explained for instance in the recent papers \cite{KO2018,KO2019}.  
Octonionic left monogenic functions do not form a right $\mathbb{O}$-module. There is no general analogue of the Borel-Pompeiu formula, either. 

The fact that one cannot interchange the parenthesis arbitrarily in a product of octonionic expressions often  represents a serious obstacle to carry over standard arguments from the Clifford analysis setting to the octonionic setting. Many times new and different ideas are required.

Very recently, Jinxun Wang and Xingmin Li started to investigate generalizations of Bergman and Hardy spaces in the octonionic setting. The study of reproducing kernel spaces represents a main research area in Clifford analysis since the 1970s, as classical reference see for example \cite{BDS}. 

In \cite{WL2018} the authors introduced an appropriate definition of an inner product that guarantees the existence and the uniqueness of a reproducing octonionic Bergman and Szeg\"o kernel over the unit ball together with explicit formulas for these kernels. The special form of this inner product is a consequence of the non-associative nature; in the associative setting, this inner product coincides with the one that is usually used in Clifford analysis. 

In \cite{WL2020} the authors developed an explicit formula for the octonionic Bergman kernel of the right octonionic half-space defined by $\{z \in \mathbb{O} \mid \Re(z) > 0\}$. To arrive at their representation formula they departed from the formula of the Bergman kernel of the unit ball developed in their preceding paper \cite{WL2018}. The octonionic setting does not offer the possibility to simply apply the Cayley transformation in the argument to directly switch from the setting of the unit ball to the half-space, because this transformation does not preserve the octonionic monogenicity. For this reason, a different approach is required to meet this end. The authors suggest two possible ways, namely by applying a limit argument passing from a sequence of balls to the entire half-space or by using a density argument which is also used in classical harmonic analysis for instance in \cite{Axler}. However, the authors have not presented yet a formula for the octonionic Szeg\"o kernel of the right-half-space. 

In our paper, we continue this line of investigation. We first develop an explicit representation formula for the octonionic Szeg\"o kernel of a strip domain of the form ${\cal{S}} := \{z \in \mathbb{O} \mid 0 < \Re(z) < d\}$ where the thickness parameter $d > 0$ can be chosen arbitrarily. This formula is expressed in terms of a singly-periodic octonionic Eisenstein type series that represents a singly-periodic octonionic monogenic generalization of the cosecant function. Then we take the limit $d \to +\infty$ and obtain an explicit formula for the octonionic Szeg\"o kernel of the right octonionic half-space, filling in the gap of the previous research mentioned above. 

Next we set up an explicit formula for the octonionic Bergman kernel of such a strip domain. Here, we use an  octonionic monogenic generalization of the cotangent function. 

Finally, we present an explicit relation between the octonionic Bergman and Szeg\"o kernel in the context of these strip domains. Like in the Clifford analysis setting, they turn out to arise from partial derivation from each other. In fact, this can be proved by applying two algebraic identities of these generalized trigonometric functions. This paper opens the door to further research in this direction.

 \section{Basic facts of octonions and octonionic function theory}
 
 The octonions represent an eight-dimensional real non-associative normed division algebra over the reals.  
Following classical references, such as \cite{Baez,WarrenDSmith} and others, the octonions can be constructed by the well-known Cayley-Dickson doubling process. This process is initiated by starting from two pairs of complex numbers $(a,b)$ and $(c,d)$. After that one forms an addition and multiplication operation by  $$
 (a,b)+(c,d) :=(a+c,b+d),\quad\quad (a,b)\cdot (c,d) := (ac-d\overline{b},\overline{a}d+cb) 
 $$ 
 where $\overline{\cdot}$ represents the conjugation (anti-)automorphism. Subsequentially, this anti-automorphism is extended by $\overline{(a,b)}:=(\overline{a},-b)$ to the set of pairs $(a,b)$. 
 
 In the first step of this doubling process, one obtains the real Hamiltonian quaternions $\mathbb{H}$. Each quaternion can be represented in the form $z=x_0 + x_ 1e_1 + x_2 e_2 + x_3 e_3$ where $e_i^2=-1$ for $i=1,2,3$ and $e_1 e_2 = e_3$, $e_2 e_3 = e_1$, $e_3 e_1 = e_2$ and $e_i e_j = - e_j e_i$ for all mutually  distinct $i,j$ from $\{1,2,3\}$. $\mathbb{H}$ is not commutative anymore. However, it is still associative.
 
In the next application of this doubling process one also looses the associativity and one deals with the Calyey octonions $\mathbb{O}$. However, in contrast to the associative Clifford algebras, octonions still form a division algebra. In real coordinates octonions can be expressed in the form  
 $$
 z = x_0 + x_1 e_1 + x_2 e_2 + x_3 e_3 + x_4 e_4 + x_5 e_5 + x_6 e_6 + x_7 e_7
 $$
 where $e_4=e_1 e_2$, $e_5=e_1 e_3$, $e_6= e_2 e_3$ and $e_7 = e_4 e_3 = (e_1 e_2) e_3$. 
 Like for quaternions, we also have $e_i^2=-1$ for all $i =1,\ldots,7$ and $e_i e_j = -e_j e_i$ for all mutual distinct $i,j \in \{1,\ldots,7\}$. The way how the octonionic multiplication works is easily visible from the table    
\begin{center}
 \begin{tabular}{|l|rrrrrrr|}
 $\cdot$ & $e_1$&  $e_2$ & $e_3$ & $e_4$ & $e_5$ & $e_6$  & $e_7$ \\ \hline
 $e_1$  &  $-1$ &  $e_4$ & $e_5$ & $-e_2$ &$-e_3$ & $-e_7$ & $e_6$ \\
 $e_2$ &  $-e_4$&   $-1$ & $e_6$ & $e_1$ & $e_7$ & $-e_3$ & $-e_5$ \\
 $e_3$ &  $-e_5$& $-e_6$ & $-1$  & $-e_7$&$e_1$  & $e_2$  & $e_4$ \\
 $e_4$ &  $e_2$ & $-e_1$ & $e_7$ & $-1$  &$-e_6$ & $e_5$  & $-e_3$\\
 $e_5$ &  $e_3$ & $-e_7$ & $-e_1$&  $e_6$&  $-1$ & $-e_4$ & $e_2$ \\
 $e_6$ &  $e_7$ &  $e_3$ & $-e_2$& $-e_5$& $e_4$ & $-1$   & $-e_1$ \\
 $e_7$ & $-e_6$ &  $e_5$ & $-e_4$& $e_3$ & $-e_2$& $e_1$  & $-1$ \\ \hline 	
 \end{tabular}
\end{center}
A particular feature of the octonions is their alternative property and that the fact that they form a composition algebra. 

We have the Moufang rule $(ab)(ca) = a((bc)a)$ holding for all $a,b,c \in \mathbb{O}$. If particularly $c=1$, then one may read off the flexibility condition $(ab)a= a(ba)$.  

Let $a = a_0 + \sum\limits_{i=1}^7 a_i e_i$ be an octonion represented with the seven imaginary units as mentioned above. $a_0$ is called the real part of $a$ and will be denoted by $\Re{a} = a_0$ throughout the whole paper. The conjugation leaves the real part invariant but it changes the minus signs of the imaginary units, i.e.  $\overline{e_j}=-e_j$ for all $j =1,\ldots,7$. Two general octonions $a,b \in \mathbb{O}$ always satisfy $\overline{a\cdot b} = \overline{b}\cdot \overline{a}$. 

The Euclidean norm and the Euclidean scalar product from $\mathbb{R}^8$ can be expressed in the octonionic context in terms of $\langle a,b \rangle := \sum\limits_{i=0}^7 a_i b_i = \Re\{a \overline{b}\}$ and $|a|:= \sqrt{\langle a,a\rangle} = \sqrt{\sum\limits_{i=0}^7 a_i^2}$. The norm composition property $|a\cdot b|=|a|\cdot|b|$ holds for all $a,b \in \mathbb{O}$. Every non-zero element $a \in \mathbb{O}$ is invertible with $a^{-1} =\overline{a}/|a|^2$.  

Another important rule that we need in this paper is the identity 
\begin{equation}\label{dieckmann}
(a\overline{b})b = \overline{b}(ba) =a(\overline{b}b)=a(b \overline{b})
\end{equation}  
which is true for all $a,b \in \mathbb{O}$ and, 
$\Re\{b(\overline{a}a)c\} =\Re\{(b \overline{a})(ac)\}$ for all $a,b,c \in \mathbb{O}$. This is explicitly proved for instance in \cite{CDieckmann} Proposition 1.6. 

We also use the notation $B_8(z,r) :=\{z \in \mathbb{O} \mid |z| < r\}$ and $\overline{B_8(z,r)} :=\{z \in \mathbb{O} \mid |z| \le r\}$ for the eight-dimensional solid open (resp.  closed) ball of radius $r$ in the octonions. By $S_7(z,r)$ we mean the seven-dimensional sphere $S_7(z,r) :=\{z \in \mathbb{O} \mid |z| = r\}$. If $z=0$ and $r=1$ then we denote the unit ball and the unit sphere by $B_8$ and $S_7$, respectively. The notation $\partial B_8(z,r)$ means the same as $S_7(z,r)$.   
 
\par\medskip\par

After having recalled the most important algebraic properties we now turn to summarize the basic analytic properties and notions in order to make the paper self-contained.  
 
From \cite{Imaeda,XL2000} and elsewhere we recall 
\begin{definition}
	Let $U \subseteq \mathbb{O}$ be open. A real differentiable function $f:U \to \mathbb{O}$ is called left (right) octonionic monogenic or equivalently left (right) ${\mathbb{O}}$-regular for short if it satisfies ${\cal{D}} f = 0$ or $f {\cal{D}} = 0$. Here $
	{\cal{D}}:= \frac{\partial }{\partial x_0} + \sum\limits_{i=1}^7 e_i \frac{\partial }{\partial x_i}$ denotes the octonionic Cauchy-Riemann operator, where $e_i$ are the octonionic units introduced above. 
\end{definition}
As already mentioned in the Introduction, the set of left (right) ${\mathbb{O}}$-regular functions does not form an ${\mathbb{O}}$-right (left) module. \cite{KO2019} presents a clear counter-example. Simply consider the function $f(z):= x_1 - x_2 e_4$. One has ${\cal{D}}[f(z)] = e_1 - e_2 e_4 = e_1 - e_1 = 0$. However, $g(z):=(f(z))\cdot e_3 = (x_1 - x_2 e_4) e_3 = x_1 e_3 - x_2 e_7$ satisfies ${\cal{D}}[g(z)] = e_1 e_3 - e_2 e_7 = e_5 -(-e_5) = 2 e_5 \neq 0$. The lack of associativity destroys the modular structure of ${\mathbb{O}}$-regular functions. This feature represents one significant difference to Clifford analysis. 

Luckily, the composition with an arbitrary translation of the form $z \mapsto z + \omega$ where $\omega \in \mathbb{O}$ still preserves monogenicity --- also in the octonionic case. We have ${\cal{D}}f(z+\omega) = 0$ if and only if ${\cal{D}}f (z) = 0$. This allows us to meaningfully introduce ${\cal{O}}$-regular generalizations of the cotangent and cosecant functions in terms of periodic Eisenstein series. This will be needed in the sequel of this paper. 
 
Like in the associative case, also ${\mathbb{O}}$-regular functions satisfy the following Cauchy integral theorem, as proved for instance in \cite{XL2002}. 
 \begin{proposition}\label{cauchy} (Cauchy's integral theorem)\\
Let $G \subseteq \mathbb{O}$ be a bounded $8$-dimensional connected star-like domain with an orientable strongly Lipschitz boundary $\partial G$. Let $f \in C^1(\overline{G},\mathbb{O})$. If $f$ is left (resp.) right $\mathbb{O}$-regular inside of $G$, then 
$$
\int\limits_{\partial G} d\sigma(z) f(z) = 0,\quad {\rm resp.}\;\;\int\limits_{\partial G} f(z) d\sigma(z) = 0
$$  	
where $d\sigma(z) = \sum\limits_{i=0}^7 (-1)^j e_i \stackrel{\wedge}{d x_i} = n(z) dS(z)$, where $\stackrel{\wedge}{dx_i} = dx_0 \wedge dx_1 \wedge \cdots dx_{i-1} \wedge dx_{i+1} \cdots \wedge dx_7$ and where $n(z)$ is the outward directed unit normal field at $z \in \partial G$ and $dS(z) =|d \sigma(z)|$ the ordinary scalar surface Lebesgue measure of the $7$-dimensional boundary surface.  
 \end{proposition}

Another essential difference to the associative setting consists in the fact that in contrast to quaternionic and Clifford analysis, octonionic analysis does {\em not} offer an analogy of a general Borel-Pompeiu formula, neither of the form 
 	$$
 	\int\limits_{\partial G} g(z) \cdot (d\sigma(z) \cdot f(z)) = 0 \quad {\rm nor}\quad \int\limits_{\partial G} (g(z) \cdot d\sigma(z)) \cdot f(z) = 0.
 	$$
 	Such a formula is not even true if $g$ is right $\mathbb{O}$-regular and $f$ left $\mathbb{O}$-regular, independently how we bracket these terms together. The lack of such an identity is already mentioned in the classical reference \cite{GTBook}.  It is a consequence of the lack of associativity. 
	
	However, if one of these functions is the Cauchy kernel $q_{\bf 0}: \mathbb{O} \backslash\{0\} \to \mathbb{O},\;q_{\bf 0}(z) := \frac{x_0 - x_1 e_1 - \cdots - x_7 e_7}{(x_0^2+x_1^2+\cdots + x_7^2)^4} = \frac{\overline{z}}{|z|^8}$ , then one obtains a generalization. 

 For convenience we recall from \cite{Imaeda,Nono,XL2002}:
 \begin{proposition}\label{cauchy1}(Cauchy' integral formula)\.\
Let $U \subseteq \mathbb{O}$ be a non-empty open set and $G \subseteq U$ be an $8$-dimensional compact oriented manifold with a strongly Lipschitz boundary $\partial G$. If $f: U \to \mathbb{O}$ is left (resp. right) $\mathbb{O}$-regular, then for all $z \not\in \partial G$
$$
\chi(z)f(z)= \frac{3}{\pi^4} \int\limits_{\partial G} q_{\bf 0}(w-z) \Big(d\sigma(w) f(w)\Big),\quad\quad \chi(z) f(z)= \frac{3}{\pi^4} \int\limits_{\partial G}   \Big(f(w)d\sigma(w)\Big) q_{\bf 0}(w-z),
$$
where $\chi(z) = 1$ if $z$ is in the interior of $G$ and $\chi(z)=0$ if $z$ in the exterior of $G$. 
 \end{proposition}
The way how the parenthesis are put is of crucial importance. Putting the parenthesis the other way around, leads in the left $\mathbb{O}$-regular case to a different formula of the form
$$
\frac{3}{\pi^4} \int\limits_{\partial G} \Big( q_{\bf 0}(w-z) d\sigma(w)\Big) f(w) = \chi(z) f(z) + \int\limits_G \sum\limits_{i=0}^7 \Big[q_{\bf 0}(w-z),{\cal{D}}f_i(w),e_i  \Big] dw_0 \wedge \cdots \wedge dw_7, 
$$
where $[a,b,c] := (ab)c - a(bc)$ stands for the associator of three octonionic elements. The volume integral which appears additionally always vanishes in algebras where one has the associativity, such as in Clifford algebras. We refer the interested reader to \cite{XL2002} for the details.

\begin{remark} It is important to mention that there is more than one possible function theory in the octonions. In the recent years, a lot of effort has been done to develop an octonionic function theory of slice regular octonionic functions, see for example {\rm \cite{GPzeroes,GP,JRS}}. An interesting research perspective will consist in developing analogies of the results from this paper for the octonionic slice regular function theories.  However, the treatment of these questions will be left for another follow-up paper.

\end{remark} 
 
\section{Singly-periodic $\mathbb{O}$-regular tangent, cotangent, secant and cosecant}
In this paper we shortly introduce the singly-periodic octonionic moogenic generalizations of the tangent, the cotangent, the secant and the cosecant function and explain some of their basic properties. In particular, we discuss their mutual relations and present a duplication formula which provides us with a generalization of the classical cotangent double angle formula. 

Like in Clifford analysis, the product or quotient of two left $\mathbb{O}$-regular functions is not left $\mathbb{O}$-regular anymore in general. From this viewpoint it does not make sense to introduce octonionic generalizations of the above mentioned trigonometric functions by forming products or quotients composed by some kind of $\mathbb{O}$-regular generalizations of the sine or cosine function. In contrast to Clifford analysis left $\mathbb{O}$-regular functions do not even have the algebraic structure of an $\mathbb{O}$-module. Fortunately, transformations of the argument of the form $z \mapsto z+ \omega$ are allowed; they keep the $\mathbb{O}$-regularity property as a consequence of the chain rule. So, one admissible way to introduce the above mentioned functions, is the use of partial fraction Mittag-Leffler series. 

Starting from the well-known cotangent identity from complex analysis
$$
\cot(z) = \frac{1}{z} + \sum\limits_{m \in \mathbb{Z} \backslash\{0\}} \Big(\frac{1}{z+\pi m} - \frac{1}{\pi m}   \Big)
$$ 
it makes sense to introduce a singly-periodic $\mathbb{O}$-regular generalized cotangent (with period $\pi$) by 
$$
\cot(z) = \sum\limits_{m \in \mathbb{Z}} q_{\bf 0}(z+ \pi m) = \sum\limits_{m \in \mathbb{Z}} \frac{\overline{z+\pi m}}{|z+\pi m|^8}. 
$$
Since $q_{\bf 0}(z)$ has a homogeneity degree of $-7$ this series can be majorized by $\sum\limits_{m \in \mathbb{N}} \frac{C}{m^7}$ which clearly is convergent. So, in contrast to the complex case, a special re-ordering of the series is not required. Since one only applies a shift in the argument of the form $z \mapsto z + \pi m$, the octonionic regularity is preserved in each term. Consequently, the $\mathbb{O}$-regularity of the cotangent follows from the direct application of the octonionic version of Weierstra{\ss}' convergence theorem presented in \cite{XZL}~Theorem~11. 

While in classical complex analysis the cotangent satisfies $2 \cot(2z) = \cot(z)+\cot(z+\frac{\pi}{2})$, in the octonionic setting it satisfies a modified duplication formula of the form
\begin{equation}\label{doubleangle}
128 \cot(2z) = \cot(z) + \cot(z+\frac{\pi}{2}),
\end{equation} 
like in the Clifford analysis case addressing $\mathbb{R}^8$.  This is a consequence of the different homogeneity degree of $q_{\bf 0}$. 

This can be verified by a direct computation of the form
\begin{eqnarray*}
& & \cot(z)+ \cot(z+\frac{\pi}{2}) -128 \cot(2z) \\
&=& \sum\limits_{m \in \mathbb{Z}} \frac{\overline{z}+\pi m}{|z+ \pi m|^8} + \sum\limits_{m \in \mathbb{Z}} \frac{\overline{z}+\frac{\pi}{2}+\pi m}{|z+\frac{\pi}{2} + \pi m|^8} - 128 \sum\limits_{m \in \mathbb{Z}} \frac{\overline{2z}+\pi m}{|2z+ \pi m|^8}  \\
&=& \sum\limits_{m \in \mathbb{Z}} \Bigg[
\frac{2^7(2\overline{z}+2\pi m)}{|2z+ 2\pi m|^8}  + \frac{2^7(2\overline{z}+\pi +2\pi m)}{|2z+\pi+ 2\pi m|^8} 
  - \frac{2^7(2\overline{z}+ \pi m)}{|2z+ \pi m|^8} \Bigg] = 0.
\end{eqnarray*}

If one wants to introduce an octonionic regular tangent function, then it does not make sense to define it as a map of the form $z \mapsto \cot(z)^{-1}$, because the quotient forming does not preserve the $\mathbb{O}$-regularity. But, like in associative Clifford analysis, one can meaningfully introduce an $\mathbb{O}$-regular tangent function by applying an $\mathbb{O}$-regular preserving linear transformation in the argument of the cotangent, viz:
$$
\tan(z) := - \cot(z+\frac{\pi}{2}).
$$
In view of the duplication formula (\ref{doubleangle}) one also gets the relation
$$
\tan(z) = \cot(z) - 128 \cot(2z).
$$
Furthermore, one can introduce an $\mathbb{O}$-regular cosecant by
$$
\csc(z) := \sum\limits_{n \in \mathbb{Z}} (-1)^n \frac{\overline{z}+\pi n}{|z+\pi n|^8} = \frac{1}{64} \cot(\frac{z}{2}) - \cot(z).
$$
Finally, the secant can be introduced in terms of a translated cosecant viz
$$
\sec(z) := \csc(z+\frac{\pi}{2}).
$$
A combination of these identies produces the relation
$$
128 \cot(2z) = \csc(z)+\tan(z)-\frac{1}{64} \tan(\frac{z}{2}).
$$

\begin{remark} As mentioned in {\rm \cite{Nolder2018}} and {\rm \cite{Kra2019-1}} one can also introduce multi-periodic generalizations of these trigonometric series, namely by extending the summation over a $p$-dimensional lattice of dimension $p \in \{2,\ldots,7\}$. In the case $p=7$ one has to re-group the terms in the way suggested in {\rm \cite{Nolder2018,Kra2019-1}} to guarantee the convergence, analogously to the classical complex and Clifford-monogenic case, see also {\rm \cite{Kra2004}}. Also in these cases one obtains generalizations of the double angle formula and the identities described above. These are similar to those obtained for the associative Clifford algebra case addressing dimension $n=8$. In this more general context one deals with octonionic monogenic generalized Eisenstein series. 
\end{remark}

\section{Octonionic Bergman and Szeg\"o kernels}
First we summarize the known results which were discovered recently:  
\subsection{The Bergman space of the octonionic unit ball}
In \cite{WL2018} the Bergman kernel of left $\mathbb{O}$-regular functions over the octonionic unit ball $B_8(0,1) = \{z \in \mathbb{O}\mid |z| < 1\}$  has been computed. 
A special difficulty in the octonionic context that had to be solved was to find the right choice of an inner product that guarantees existence and uniqueness of the Bergman kernel in the unit ball setting. Instead of defining the inner product in the classical way, the authors proposed the following modified definition of an $L^2$-inner product. For two functions $f,g \in {\cal{B}}^2(B_8(0,1)):= L^2(B_8(0,1)) \cap {\rm Ker} \;{\cal{D}}$ which means that
$$
\int\limits_{B_8(0,1)} |f(z)|^2 dV(z) < +\infty, \quad\quad {\cal{D}}f(z)=0\;\;\forall z \in B_8(0,1)
$$
and 
$$
\int\limits_{B_8(0,1)} |g(z)|^2 dV(z) < +\infty, \quad\quad {\cal{D}}g(z)=0\;\;\forall z \in B_8(0,1),
$$
where $dV = dx_0 dx_1 \cdots dx_7$ is the scalar valued volume measure, 
the Bergman inner product is defined by
$$
(f,g)_{B_8} := \frac{3}{\pi^4} \int\limits_{B_8(0,1)} \Big(\overline{g(z)} \cdot  \frac{\overline{z}}{|z|}   \Big) \cdot \Big(\frac{z}{|z|} \cdot f(z)  \Big) \cdot dV.
$$
Notice that in the case of associativity, this inner product coincides with the usual one used in Clifford analysis, since $\frac{\overline{z}}{|z|} \cdot \frac{z}{|z|}=1$. However, in the octonionic case, the parenthesis cannot be shifted, so it makes a difference. Actually, it is this particular choice that guarantees the existence of a unique Bergman kernel. If $f=g$, then the inner product represents the usual $L_2$ norm introduced before, because $$(f,f) = \frac{3}{\pi^4} \int\limits_{B_8(0,1)} \overline{f(z)} f(z) dV = \frac{3}{\pi^4} \int\limits_{B_8(0,1)} |f(z)|^2 dV =:\|f\|^2_{B_8},$$ in view of the octonionic identity $$(\overline{f} \overline{z}) \cdot (z f) =  \overline{f} (\overline{z}z) f = \overline{f} |z|^2 f = |f|^2 |z|^2.$$

In \cite{WL2018} the authors showed by an explicit computation that the expression
$$
B_{B_8}(z,w) = \frac{[6(1-|w|^2|z|^2)+2(1-\overline{z}w)]\cdot (1-\overline{z}w)
}{|1-\overline{z}w|^{10}} 
$$
reproduces all elements of ${\cal{B}}^2(B_8(0,1))$ and hence is the $\mathbb{O}$-regular Bergman kernel of $B_8(0,1)$. One has $f(w)=(f,B(\cdot,w))_{B_8}$. In Paragraph 5 of \cite{WL2018} the authors also proved the more compact representation 
$$
\overline{B_{B_8}(z,w)} \cdot \overline{z} = \overline{{\cal{D}}_w} \Bigg( \frac{1-|w|^2|z|^2}{|1-w\overline{z}|^8}\Bigg).
$$   
Notice that in contrast to the Clifford analysis setting, the octonionic Bergman space of left $\mathbb{O}$-regular functions does not have the algebraic structure of a right $\mathbb{O}$-module, since $f \in $Ker ${\cal{D}}$ does not imply that also $f\cdot \lambda \in $ Ker ${\cal{D}}$ for a general octonionic $\lambda$. This is a further difference to Clifford analysis. The term space here does not have the meaning of a classical vector space in the context of this paper. The octonionic Bergman space is neither an octonionic vector space nor even an octonionic module. The same holds for the octonionic Hardy space, that will be introduced in the next subsection.  
\subsection{The Hardy space and the Szeg\"o kernel of the octonionic unit ball}
In \cite{WL2018} a formula for the reproducing Szeg\"o kernel of the octonionic unit ball has been presented, too. The associated Hardy space $H^2(S_7)$ of left $\mathbb{O}$-regular functions is said to be the closure of the set of all functions that are square-integrable over the unit sphere and left $\mathbb{O}$- regular inside the unit ball and that have additionally a continuous extension to the boundary. Here, square-integrable means that 
$$
\int\limits_{S_7(0,1)} |f(z)|^2 dS,
$$ 
where $dS$ is the scalar surface measure of the $7$-dimensional unit sphere in $\mathbb{R}^8$. 

Similarly to the Bergman case it makes sense to introduce the inner product of two functions $f,g \in H^2(S_7)$ by 
$$
(f,g)_{S_7} := \frac{3}{\pi^4} \int\limits_{S_7(0,1)} \overline{(z \cdot g(z))} \cdot (z \cdot f(z)) dS = \frac{3}{\pi^4} \int\limits_{S_7(0,1)} (\overline{g(z)} \cdot \overline{z}) \cdot (z \cdot f(z)) dS,
$$
cf. \cite{WL2018}. Defining the inner product in this way, one can directly deduce the associated Szeg\"o kernel from Cauchy's integral formula extending the classical argumentation from the Clifford analysis setting, cf. for instance \cite{Co}. Note that on the unit sphere the oriented octonionic surface element $d\sigma(z)$ can be written explicitly as $d\sigma(z) = {\bf n}(z) \cdot dS = z \cdot dS$ where the exterior unit normal field ${\bf n}(z)=z$ in the octonionic language. 

For any $f \in H^2(S_7)$ Proposition~\ref{cauchy} therefore implies that 
\begin{eqnarray*}
f(w) &=& \frac{3}{\pi^4} \int\limits_{S_7} \frac{\overline{z-w}}{|z-w|^8} \cdot \Bigg( d\sigma(z) \cdot f(z) \Bigg)\\
&=& \frac{3}{\pi^4} \int\limits_{S_7} \frac{\overline{z-w}}{|z-w|^8} \cdot \Bigg( z \cdot dS \cdot f(z) \Bigg)
\end{eqnarray*}
Since $dS$ is scalar valued, it commutes with the octonionic expressions, and the latter equation reads
\begin{eqnarray*}
f(w) &=& \frac{3}{\pi^4} \int\limits_{S_7} \frac{\overline{z-w}}{|z-w|^8} \cdot (z \cdot f(z)) dS\\
&=& \frac{3}{\pi^4} \int\limits_{S_7} \frac{\overline{z}-\overline{w}(z \overline{z})}{\underbrace{|z|}_{=1}|1-\overline{z}w|^8} \cdot (z \cdot f(z)) dS.
\end{eqnarray*}
By applying the rule (\ref{dieckmann}) in the numerator and relying on $|z|=1$ in the denominator the latter equation becomes 
\begin{eqnarray*}
f(w) &=& \frac{3}{\pi^4} \int\limits_{S_7} \frac{\overline{z}-(\overline{w}\cdot z)\cdot \overline{z}}{|1-\overline{z}w|^8} \cdot (z \cdot f(z)) dS\\
&=& \frac{3}{\pi^4} \int\limits_{S_7} \frac{(1-\overline{w}z)\cdot \overline{z}}{|1-\overline{z}w|^8} \cdot (z \cdot f(z)) dS\\
&=& \frac{3}{\pi^4} \int\limits_{S_7} \Bigg(\overline{\frac{1-\overline{z}w}{|1-\overline{z}w|^8}} \cdot \overline{z}       \Bigg) \cdot \Bigg(z \cdot f(z)     \Bigg) dS
\end{eqnarray*}
Taking into account the definition of the inner product allows us to read off the reproducing Szeg\"o kernel as
$$
S_{S_7}(z,w) = \frac{1-\overline{z}w}{|1-\overline{w}|^8},
$$ 
satisfying the reproduction property $f(w)=(f,S(\cdot,w))_{S_7}$ in the sense of this inner product.  

\subsection{The Bergman kernel of the octonionic right half-space}
In their very recent paper \cite{WL2020} Jinxun Wang and Xingmin Li also established a formula for the Bergman kernel of the octonionic right half-space 
$$
\mathbb{R}^8_+ = H^{+}(\mathbb{O}) := \{z \in \mathbb{O} \mid \Re(z) > 0\}
$$
In this setting one considers functions $f,g \in L^2(H^{+}(\mathbb{O})) \cap {\rm Ker}\;{\cal{D}} =:B^2(H^{+}(\mathbb{O}))$, i.e. one requires these functions to satisfy
$$
\int\limits_{H^{+}(\mathbb{O})} |f(z)|^2 dV < +\infty,\quad \int\limits_{H^{+}(\mathbb{O})} |g(z)|^2 dV < +\infty
$$
and ${\cal{D}}f(z) = 0$ and ${\cal{D}} g(z) = 0$ for all $z \in H^{+}(\mathbb{O})$. In contrast to the unit ball setting there is no obstacle to define the inner product on $B^2(H^{+}(\mathbb{O}))$ in the usual way by 
$$
(f,g)_{H^+} := \frac{3}{\pi^4} \int\limits_{H^{+}(\mathbb{O})} \overline{g(z)} \cdot f(z) dV(z),
$$ 
cf. \cite{WL2020}. Following \cite{WL2020} Theorem~1, in the sense of this inner product the Bergman kernel of $B^2(H^{+}(\mathbb{O}))$ exists and is uniquely defined. It equals 
$$
B_{H^+}(z,w) = - 2 \frac{\partial }{\partial x_0} q_{\bf 0}(z + \overline{w}) = - 2 \frac{\partial }{\partial x_0} \Bigg\{\frac{\overline{z}+w}{|\overline{z}+w|^8}   \Bigg\} 
$$
and actually satisfies $f(w) =(f,B_{H^+}(\cdot,w))_{H^+}$ in the sense of this inner product. 
\begin{remark} The authors suggest two different ways of proving this formula in {\rm \cite{WL2020}}. 
\par\medskip\par
1. Method: One starts from the formula for the Bergman kernel of the unit ball. First one shows applying the known formula (relying on the particular choice of the inner product in the unit ball setting as described above) that the $\mathbb{O}$-regular Bergman kernel of an arbitrary ball $B(p,r)$ centered around a fixed $p \in \mathbb{O}$ with arbitrary radius $r > 0$ has the form
$$
B_{B_8(p,r)}(z,w) = \frac{1}{r^8} B_{B_8(0,1)}\Big(\frac{z-p}{r},\frac{w-p}{r}\Big).
$$ 
In this context the proper choice of the inner product for $B^2(B_{p,r})$ is
$$
(f,g)_{B_8(p,r)} := \frac{3}{\pi^4} \int\limits_{B_8(p,r)} \Bigg( \overline{g(z)} \cdot \frac{\overline{z-p}}{|z-p|}   \Bigg) \cdot \Bigg(\frac{z-p}{|z-p|} \cdot f(z)   \Bigg) dV(z).
$$
Next the authors apply a limit argument ({\rm Lemma 3.3  \cite{WL2020}}) establishing that 
$$
B_{H^+}(z,w) = \lim\limits_{r \to \infty} B_{B_8(r,r)}(z,w) \quad \forall z,w \in H^+(\mathbb{O})
$$
from which the formula follows.
\par\medskip\par
2. Method: One can also use a classical density argument. 

For $\delta > 0$ one defines $H^{\delta} := \{z \in \mathbb{O} \mid \Re(\delta) > 0\}$ and considers $B^2(H^{\delta}) := L^2(H^{\delta}) \cap {\rm Ker}\; {\cal{D}}$. With the same standard arguments that one uses in classical harmonic analysis the authors managed ({\rm Lemma~1, \cite{WL2020}}) to prove that $\bigcup_{0 < \delta < 1} B^2(H^{\delta})$ is dense in $B^2(H^+)$, see {\rm \cite{Axler}}. This tool in hand allows one to prove the reproduction of the kernel expression directly by performing partial integration with respect to the $x_0$-direction, in the same way as presented by J. Cnops in {\rm \cite{Cn}} for the Clifford analysis setting.
\end{remark} 
\subsection{The Szeg\"o kernel of an octonionic strip domain and of the right half-space}
In this section we set up an explicit formula for the $\mathbb{O}$-regular Szeg\"o kernel of an octonionic strip domain of the form 
$$
{\cal{S}} := \{ z \in \mathbb{O} \mid 0 < \Re(z) < d\} 
$$ 
where $d > 0$ is an arbitrary real parameter. 

The inner product in the context of the octonionic Hardy space over ${\cal{S}}$ has the form
$$
(f,g)_{\partial {\cal{S}}} := \frac{3}{\pi^4} \int\limits_{\partial {\cal{S}}} \overline{g(z)} f(z) dS(z).
$$

This consideration then also allows us to set up a formula for the Szeg\"o kernel of the right half-space $H^{+}(\mathbb{O})$ by considering the limit $d \to +\infty$ which fills in a gap that was left open in the preceeding papers \cite{WL2018,WL2020}. In addition to that the results presented in this section and the following one provide us with an octonionic generalization of the results that we obtained in the associative Clifford analysis setting in \cite{ConKraCV}. 

First we prove that 
\begin{theorem}
The $\mathbb{O}$-regular Szeg\"o kernel of the strip domain ${\cal{S}} := \{ z \in \mathbb{O} \mid 0 < \Re(z) < d\}$ has the form
$$
S_{\cal{S}}(z,w) =  \sum\limits_{n=-\infty}^{+\infty} (-1)^n \frac{\overline{z}+w+2dn}{|\overline{z}+w+2dn|^8} = \Big(\frac{\pi}{2d}\Big)^7 \csc(\frac{\pi}{2d}(z+\overline{w})).
$$
\end{theorem} 
\begin{proof}
Consider first a function $f$ is that is left $\mathbb{O}$-regular in a neighborhood of the closure of ${\cal{S}}$. The boundary $\partial {\cal{S}}$ splits into two components: 
$$
\partial {\cal{S}} = \{ z \in \mathbb{O} \mid x_0=0\} \cup \{z \in \mathbb{O} \mid x_0 = d\} =: \partial {\cal{S}}_1 \cup \partial {\cal{S}}_2.
$$
On the component $\partial {\cal{S}}_1$ we have $\Re(w) = w_0 = 0$ which means $\overline{w}=-w$. On the other component $\partial {\cal{S}}_2$ we have $\Re(w)=d$ which means that $\overline{w}=2d-w$. Next one computes
\begin{eqnarray*}
 & & \int\limits_{w \in \partial {\cal{S}}} \frac{3}{\pi^4} \Bigg(\sum\limits_{n=-\infty}^{+\infty} (-1)^n \frac{\overline{z}+w+2dn}{|\overline{z}+w+2dn|^8} \Bigg) \cdot f(w) \cdot dS(w)\\
&=& \int\limits_{w \in \partial {\cal{S}}_1} \frac{3}{\pi^4} \Bigg(\sum\limits_{n=-\infty}^{+\infty} (-1)^n \frac{\overline{z}+w+2dn}{|\overline{z}+w+2dn|^8} \Bigg) \cdot f(w) \cdot dS(w)\\
&+& \int\limits_{w \in \partial {\cal{S}}_2} \frac{3}{\pi^4} \Bigg(\sum\limits_{n=-\infty}^{+\infty} (-1)^n \frac{\overline{z}+w+2dn}{|\overline{z}+w+2dn|^8} \Bigg) \cdot f(w) \cdot dS(w) \\
&=& \int\limits_{w \in \partial {\cal{S}}_1} \frac{3}{\pi^4} \Bigg(\sum\limits_{n=-\infty}^{+\infty} (-1)^n \frac{\overline{z+2dn-w}}{|z+2dn-w|^8} \Bigg) \cdot f(w) \cdot dS(w)\\
&+& \int\limits_{w \in \partial {\cal{S}}_2} \frac{3}{\pi^4} \Bigg(\sum\limits_{n=-\infty}^{+\infty} (-1)^n \frac{\overline{z+2dn+2d-w}}{|z+2dn-2d-w|^8} \Bigg) \cdot f(w) \cdot dS(w)\\
&=& \int\limits_{w \in \partial {\cal{S}}_1} \frac{3}{\pi^4} \Bigg(\sum\limits_{n=-\infty}^{+\infty} (-1)^n \frac{\overline{z+2dn-w}}{|z+2dn-w|^8} \Bigg) \cdot f(w) \cdot dS(w)\\
&+& \int\limits_{w \in \partial {\cal{S}}_2} \frac{3}{\pi^4} \Bigg(\sum\limits_{n=-\infty}^{+\infty} (-1)^n \frac{\overline{z+2d(n+1)-w}}{|z+2d(n+1)-w|^8} \Bigg) \cdot f(w) \cdot dS(w)\\
&=& \int\limits_{w \in \partial {\cal{S}}_1} \frac{3}{\pi^4} \Bigg(\sum\limits_{n=-\infty}^{+\infty} (-1)^n \frac{\overline{z+2dn-w}}{|z+2dn-w|^8} \Bigg) \cdot f(w) \cdot dS(w)\\
&+& \int\limits_{w \in \partial {\cal{S}}_2} \frac{3}{\pi^4} \Bigg(\sum\limits_{N=-\infty}^{+\infty} (-1)^{N+1} \frac{\overline{z+2dN-w}}{|z+2dN-w|^8} \Bigg) \cdot f(w) \cdot dS(w)\\
&=& \int\limits_{w \in \partial {\cal{S}}_1} \frac{3}{\pi^4} \Bigg(\sum\limits_{n=-\infty}^{+\infty} (-1)^n \frac{\overline{z+2dn-w}}{|z+2dn-w|^8} \Bigg) \cdot f(w) \cdot dS(w)\\
&-& \int\limits_{w \in \partial {\cal{S}}_2} \frac{3}{\pi^4} \Bigg(\sum\limits_{n=-\infty}^{+\infty} (-1)^n \frac{\overline{z+2d-w}}{|z+2dn-w|^8} \Bigg) \cdot f(w) \cdot dS(w).
\end{eqnarray*} 
Note that the exterior unit normal at $\partial {\cal{S}}_1$ equals ${\bf n} =-e_0$ but at $\partial {\cal{S}}_2$ it equals ${\bf n}=+e_0$. The previous sum of the two integrals can thus be written as one single integral in the form 
\begin{eqnarray*}
 & & \int\limits_{w \in \partial {\cal{S}}_1} \frac{3}{\pi^4} \Bigg(\sum\limits_{n=-\infty}^{+\infty} (-1)^n 
\frac{\overline{z}-\overline{w}+2dn}{|z-w+2dn|^8}  \Bigg) f(w) dS(w)\\
&-&\int\limits_{w \in \partial {\cal{S}}_2} \frac{3}{\pi^4} \Bigg(\sum\limits_{n=-\infty}^{+\infty} (-1)^n 
\frac{\overline{z}-\overline{w}+2dn}{|z-w+2dn|^8}  \Bigg) f(w) dS(w)\\
&=& - \int\limits_{w \in \partial {\cal{S}}} \frac{3}{\pi^4} \Bigg(\sum\limits_{n=-\infty}^{+\infty} (-1)^n 
\frac{\overline{z}-\overline{w}+2dn}{|z-w+2dn|^8}  \Bigg)\cdot (d\sigma(w) \cdot f(w))\\
&=&  \int\limits_{w \in \partial {\cal{S}}} \frac{3}{\pi^4} \Bigg(\sum\limits_{m=-\infty}^{+\infty} (-1)^m  
\frac{\overline{w}-\overline{z}+2dm}{|w-z+2dm|^8}  \Bigg)\cdot (d\sigma(w) \cdot f(w)).
\end{eqnarray*}
because $d\sigma(w) = {\bf n} dS = \left\{ \begin{array}{rl} -e_0 dS, & {\rm for}\;w_0 = 0\\ e_0 dS, & {\rm for}\;w_0 = d    \end{array} \right.$

Note that $\overline{z}-\overline{w}+2dn = 0$ if and only if $\overline{w}-\overline{z}=2dn$ which only vanishes if $n=0$. The singularity only may occure in the term associated with $n=0$. Thus, the series over the integrals reduces to
\begin{eqnarray*}
& & \frac{3}{\pi^4} \int\limits_{\partial {\cal{S}}}\frac{\overline{w-z}}{|w-z|^8} \cdot \Bigg( d\sigma(w) \cdot f(w) \Bigg)\\
&+ & \frac{3}{\pi^4} \int\limits_{\partial {\cal{S}}} \Bigg(\sum\limits_{n=-\infty, n \neq 0}^{+\infty}(-1)^n
\frac{\overline{w}-\overline{z}+2dn}{|w-z+2dn|^8}    \Bigg)     \cdot \Bigg( d\sigma(w) \cdot f(w) \Bigg)\\
&=& f(z)+0 = f(z),
\end{eqnarray*}
where we applied in the first time the Cauchy integral formula. The Cauchy integral formula is also applied in the second part. Putting $\zeta := z + 2dn$, each term of the second part has the form
$$
\frac{3}{\pi^4} \int\limits_{\partial{\cal{S}}}\frac{\overline{w-\zeta}}{|w-\zeta|^8} \cdot (d\sigma(w)\cdot f(w))
$$
which vanishes completely according to Proposition~\ref{cauchy1}, because $\zeta$ lies outside of ${\cal{S}}$. $\zeta = w$ can never occure whenever $n \neq 0$ and whenever $w \in int({\cal{S}})$. It is important to mention at this point that our argumentation really requires the explicit structure of the Cauchy integral formula, because in the octonionic setting we cannot rely on a general Borel-Pomepeiu formula. Only if $g = q_{\bf 0}$, then this argumentation is true in this generality (putting the parenthesis in exactly this way). 
The rest of the statement follows by a standard stretching argument, applied in the same way as performed in \cite{ConKraCV} for the associative Clifford algebra setting. 
  
\end{proof}
From this formula it is now very easy to deduce 
\begin{corollary}
The $\mathbb{O}$-regular Szeg\"o kernel of the right octonionic half-space $H^{+}(\mathbb{O})$ has the form
$$
S_{H^{+}}(z,w) =  \frac{\overline{z}+w}{|\overline{z}+w|^8}.
$$
\end{corollary}
\begin{proof}
In the previous theorem we proved that the Szeg\"o kernel of a strip domain bounded by the surfaces $x_0=0$ and $x_0=d$ has the representation:
\begin{eqnarray*}
S_{\cal{S}}(z,w) &=&  \sum\limits_{n=-\infty}^{+\infty} (-1)^n \frac{\overline{z}+w+2dn}{|\overline{z}+w+2dn|^8}\\
&=&  \Bigg[ \frac{\overline{z}+w}{|\overline{z}+w|^8} + \sum\limits_{n \neq 0} (-1)^n \frac{\overline{z}+w+2dn}{|\overline{z}+w+2dn|^8}\Bigg].
\end{eqnarray*}
In the limit case $d \to +\infty$, which then represents the right half-space case, we have 
$$
\frac{\overline{z}+w+2dn}{|\overline{z}+w+2dn|^8} \sim  \frac{C}{d^7} \to 0.
$$
from which the statement follows.
\end{proof}.
\begin{remark} By a direct comparison we observe that in analogy to the complex and Clifford analysis case there is a direct relation between the Szeg\"o kernel and the Bergman kernel of the right half-space of the explicit form
$$
B_{H^+}(z,w) = - 2  \frac{\partial }{\partial x_0} S_{H^+}(z,w).
$$  
\end{remark}

\subsection{The Bergman kernel of octonionic strip domains and its relation to the Szeg\"o kernel}
Finally we round off this paper by also proving a formula for the $\mathbb{O}$-regular Bergman kernel of a  strip domain of the form ${\cal{S}} = \{z \in \mathbb{O} \mid 0 < \Re(z) < d\}$ of general width $d > 0$. For clarity, the Bergman inner product in this context is defined by
$$
(f,g)_{{\cal{S}}} := \frac{3}{\pi^4} \int\limits_{{\cal{S}}} \overline{g(z)} f(z) dV(z).
$$
We prove
\begin{theorem} The reproducing $\mathbb{O}$-regular Bergman kernel of the strip domain ${\cal{S}}$ with respect to the inner product presented above has the form
$$
B_{{\cal{S}}}(z,w) = - 2 \frac{\partial }{\partial x_0} \Big\{ \Big(\frac{\pi}{2d}\Big)^7 \cot(\frac{\pi}{d}(z+\overline{w}))  \Big\} = (- 2) \sum\limits_{n=-\infty}^{+\infty} \frac{\partial }{\partial x_0} \Big\{    
\frac{\overline{z}+w+2dn}{|\overline{z}+w+2dn|^8}.
\Big\}
$$
\end{theorem}
\begin{proof}
Suppose first that $f$ is some function being left $\mathbb{O}$-regular in a neighborhood of ${\cal{S}}$. Then we have
\begin{eqnarray*}
& & \frac{3}{\pi^4} \int\limits_{w \in {\cal{S}}} \overline{B_{\cal{S}} (w,z)} f(w) dV(w)\\
&=& \frac{3}{\pi^4} \int\limits_{w \in {\cal{S}}} (-2) \frac{\partial }{\partial w_0} \Bigg\{
\sum\limits_{n=-\infty}^{+\infty} \frac{\overline{z}+w+2dn}{|\overline{z}+w+2dn|^8}
\Bigg\} dV(w) f(w)\\
&=& \frac{3}{\pi^4} \Bigg( \int\limits_{w \in {\cal{S}},w_0=0} - \int\limits_{w \in {\cal{S}},w_0=d} \Bigg)
\Bigg(\sum\limits_{n=-\infty}^{+\infty} \frac{\overline{z}+w+2dn}{|\overline{z}+w+2dn|^8}    \Bigg) dS(w) f(w)
\end{eqnarray*} 
where we have applied partial integration with respect to $w_0$. In view of $w=-\overline{w}$ on $w_0=0$ and $w = -\overline{w}+2d$, the latter integrals can be re-written in the following equivalent way
\begin{eqnarray*}
& & \frac{3}{\pi^4} \Bigg( \int\limits_{w \in {\cal{S}},w_0=0} - \int\limits_{w \in {\cal{S}},w_0=d} \Bigg)
\Bigg(\sum\limits_{n=-\infty}^{+\infty} \frac{\overline{z}-\overline{w}+2dn}{|\overline{z}-\overline{w}+2dn|^8}    \Bigg) dS(w) f(w)\\
&=& - \frac{3}{\pi^4} \int\limits_{\partial {\cal{S}}} \Bigg( \sum\limits_{n=-\infty}^{+\infty} \frac{\overline{z}-\overline{w}+2dn}{|\overline{z}-\overline{w}+2dn|^8} \Bigg)\cdot (d\sigma(w)\cdot f(w)),
\end{eqnarray*}
since again $d\sigma(w) = \left\{ \begin{array}{rl} -e_0 dS, & {\rm for}\;w_0 = 0\\ e_0 dS, & {\rm for}\;w_0 = d    \end{array} \right.$. So, the latter integral equals 
$$
\frac{3}{\pi^4} \int\limits_{\partial {\cal{S}}} \frac{\overline{w-z}}{|w-z|^8} \cdot (d\sigma(w)\cdot f(w)) 
- \underbrace{\frac{3}{\pi^4} \int\limits_{\partial {\cal{S}}} \Bigg( \sum\limits_{n \neq 0}  \frac{\overline{z-w}+2dn}{|z-w+2dn|^8}\Bigg) \cdot (d\sigma(w)\cdot f(w))}_{=0} = f(z), 
$$
where we again applied the octonionic Cauchy integral formula.

The reproduction property for all functions $f \in L^2({\cal{S}}) \cap {\rm Ker}\;{\cal{D}}$ finally follows after applying a standard streching argument. 
\end{proof}
In view of the octonionic trigonometric identities
$$
\csc(z) = \frac{1}{64} \cot(\frac{z}{2}) - \cot(z)
$$
and
$$
128 \cot(2z) = \cot(z) + \cot(z+\frac{\pi}{2})
$$
presented in Section~3, we can deduce the following interesting relation between both kernel functions
$$
B_{{\cal{S}}}(\frac{z}{2},\frac{w}{2}) - B_{{\cal{S}}}(\frac{z+d}{2},\frac{w+d}{2}) = -512 \frac{\partial }{\partial x_0} S_{{\cal{S}}}(z,w).
$$
In the limit case $d \to +\infty$ one re-obtains the previously given formula for the Bergman kernel of the octonionic right half-space, namely
\begin{eqnarray*}
B_{H^{+}}(z,w) &=& -2 \frac{\partial }{\partial x_0} \Bigg\{    
\frac{\overline{z}+w}{|\overline{z}+w|^8} + \sum\limits_{n \neq 0} \underbrace{\frac{\overline{z}+w+2dn}{|\overline{z}+w+2dn|^8}}_{\to 0,\; {\rm if}\;d \to +\infty}\Bigg\}\\
&=& -2 \frac{\partial }{\partial x_0} \Bigg\{    
\frac{\overline{z}+w}{|\overline{z}+w|^8}
\Bigg\}.
\end{eqnarray*}
\begin{remark}
In the associative Clifford analysis setting we could even set up formulas for the Bergman kernel for more general rectangular domains bounded in several directions, see \cite{ConKraMMAS}. However, when switching from $d\sigma(z)$ to $dS(z)$ the non-associativity provides us with a further obstacle when the unit normal is not a real number (which luckily is the case in the particular strip domains treated in this paper). If we want to address this more general context, then new ideas are required. Candidates for the reproducing kernels then are the multiperiodic variants of the cotangent functions where the summation is extended over a multi-dimensional lattices as roughly indicated at the end of Section~3.  
\end{remark}


\begin{thebibliography}{99}
 
\bibitem{Axler} S. Axler, P. Bourdon, W. Ramey. {\it Harmonic Function Theory} (2nd edition), Springer, New York, 2001.

\bibitem{Baez} J. Baez. The octonions, {\it Bull. Amer. Math. Soc.} {\bf 39} (2002), 145--205. 

\bibitem{BDS} F. Brackx, R. Delanghe, F. Sommen. {\it Clifford Analysis},
Pitman Res. Notes in Math. {\bf 76}, Boston, 1982.

\bibitem{Burdik} C. Burdik, S. Catto, Y. G\"urcan, A. Khalfan, L. Kurt, V. Kato La. $SO(9,1)$ Group and Examples of Analytic Functions,  {\it Journal of Physics: Conf. Series} {\bf 1194} (2019), 012016.
  
 \bibitem{CDieckmann} C. Dieckmann. {\it Jacobiformen \"uber den Cayley-Zahlen}, PhD Thesis, Lehrstuhl A f\"ur Mathematik, 2014, https://publications.rwth-aachen.de/record/445009/files/5202.pdf

\bibitem{Cn} J. Cnops. {\it Hurwitz Pairs and Applications of M\"obius Transformations}, Habilitation Thesis, Faculteit van de Wetenschappen, Rijksuniversiteit Gent, Academiejaar 1993-1994.

\bibitem{Co} D. Constales. {\it The relative position of $L^2$ domains in complex and Clifford analysis}. PhD Thesis, Faculteit van de Wetenschappen, Rijksuniversiteit Gent, Academiejaar 1989-1990. 

\bibitem{ConKraMMAS} D. Constales, R.S. Krau{\ss}har. Bergman kernels for rectangular domains and multiperiodic functionsin Clifford analysis, {\it Mathematical Methods in the Applied Sciences} {\bf 25} (2002), 1509--1526.

\bibitem{ConKraCV} D. Constales, R.S. Krau{\ss}har. Szeg\"o and Polymonogenic Bergman Kernels for Half-Space and Strip Domains, and Single-Periodic Functions in Clifford Analysis, {\it Complex Variables} {\bf 47} No. 4 (2002), 349--361.
 
 \bibitem{CSS} F. Colombo, I. Sabadini, D. Struppa. {\it Entire slice regular functions}, Springer, 2016, Cham. 

\bibitem{DS} P. Dentoni, M. Sce. Funzioni regolari nell'algebra di Cayley, {\it Rend. Sem. Mat. Univ. Padova} {\bf 50} (1973), 251--267.

\bibitem{DM} T. Dray, C. Manogue. {\it The Geometry of the Octonions}, World Scientific, Singapore, 2015, 228pp. https://doi.org/10.1142/8456 
   
\bibitem{GPzeroes} R. Ghiloni, A. Perotti. Zeros of regular functions of quaternionic and octonionic variable: a division lemma and the camshaft effect. {\it Annali di Matematica Pura ed Applicata} {\bf 190} (2011), 539–-551.
 
\bibitem{GP} R. Ghiloni, A. Perotti.  Slice regular functions on real alternative	algebras. {\it Adv. Math.}, {\bf 226} (2011), 1662--1691. 

\bibitem{GTBook} F. G\"ursey, H. Tze. {\it  On the role of division and Jordan algebras in particle physics}, World Scientific, Singapore, 1996.
 
 \bibitem{Imaeda} K. Imaeda.  Sedenions: algebra and analysis. {\it Appl. Math. Comp.} {\bf 115} (2000), 77--88.
 
 \bibitem{JRS} Ming Jin, Guangbin Ren, I. Sabadini. Slice Dirac operator over octonions. To appear in {\it Israel J. Math.}, https://arxiv.org/pdf/1908.01383.pdf
 
\bibitem{KO2018} J. Kauhanen, H. Orelma. Cauchy-Riemann Operators in Octonionic Analysis, {\it Advances in Applied Clifford Algebras} {\bf 28} No. 1 (2018), 14pp.

\bibitem{KO2019} J. Kauhanen, H. Orelma. On the structure of Octonion regular functions, {\it Advances in Applied Clifford Algebras} {\bf 29} No. 4 (2019), 17pp. 

\bibitem{Kra2004} R.S. Krau{\ss}har. {\it Generalized automorphic Forms in hypercomplex spaces}. Birkh\"auser, Basel, 2004.

\bibitem{Kra2019-1} R.S. Krau{\ss}har. Function Theories in Cayley-Dickson algebras and Number Theory, submitted for publication (2019), 19pp,  https://arxiv.org/abs/1912.01351

\bibitem{Kra2019-2} R.S. Krau{\ss}har. Conformal mappings revisited in the octonions and Clifford algebras of arbitrary dimension, {\it Advances in Applied Clifford Algebras} {\bf 30}:36 (2020), 14pp. 

\bibitem{Kra2020-1} R.S. Krau{\ss}har. Differential topological aspects in octonionic
monogenic function theory, {\it Advances in Applied Clifford Algebras} {\bf 30}:51 (2020), 25pp. 


\bibitem{Nolder2018} C. Nolder. Much to do about octonions. {\it AIP Proceedings}, ICNAAM 2018, 2116, 160007 (2019); https://doi.org/10.1063/1.5114151
 
\bibitem{Nono} K. Nono. On the octonionic linearization of Laplacian and octonionic function theory, {\it Bull. Fukuoka Univ. Ed. Part} III {\bf 37} (1988), 1-–15.

\bibitem{WarrenDSmith} W. D. Smith. Quaternions, octonions, 16-ons and $2^n$-ons; New kinds of numbers, Pensylvania State University (2004), 1--68. DOI: 10.1.1.672.2288

\bibitem{XL2000} Xing-min Li and Li-Zhong Peng.  On Stein-Weiss conjugate harmonic function and octonion analytic function, {\it Approx. Theory and its Appl.} {\bf 16} (2000), 28–-36.

\bibitem{XL2001} Xing-min Li, K. Zhao, Li-Zhong Peng, The Laurent series on the octonions. Adv. Appl.Clifford Alg. {\bf 11} (S2) (2001), 205-–217.

\bibitem{XL2002} Xing-min Li and Li-Zhong Peng. The Cauchy integral formulas on the octonions, {\it Bull. Belg. Math. Soc.} {\bf 9}   (2002), 47--62.  
 
\bibitem{XZL} Xing-min Li , Zhao Kai, Li-Zhong Peng. Characterization of octonionic analytic functions, {\it Complex Variables} {\bf 50} No. 13 (2005), 1031--1040. 

\bibitem{WL2018} Jinxun Wang and Xingmin Li. The octonionic Bergman kernel for the unit ball, {\it Adv. Appl. Clifford Algebras} {\bf 28}:60 (2018).

\bibitem{WL2020} Jinxun Wang and Xingmin Li. The octonionic Bergman kernel for the half space, to appear in {\it Adv. Appl. Clifford Algebras}, https://arxiv.org/abs/1910.01478

\end{thebibliography}
\end{document}